\documentclass[11pt]{amsart}
\usepackage[top=30truemm,bottom=30truemm,left=30truemm,right=30truemm]{geometry} 
\usepackage{amsmath, amssymb, array}
\usepackage[all]{xy}
\usepackage{ascmac}
\usepackage[dvips]{graphicx}
\usepackage{color}
\usepackage{amscd}
\usepackage{fancyhdr}
\usepackage[driverfallback=dvipdfm]{hyperref}
\usepackage{amsrefs}
\usepackage{cleveref}
\usepackage{listings}

\newcommand{\bm}[1]{{\mbox{\boldmath $#1$}}}
\newcommand{\setmid}{\mathrel{}\middle|\mathrel{}}

\title{Counterexamples to the local-global principle for non-singular plane curves
and a cubic analogue of Ankeny-Artin-Chowla-Mordell conjecture}
\author{Yoshinosuke Hirakawa}
\author{Yosuke Shimizu}
\address[Yoshinosuke Hirakawa, Yosuke Shimizu]{Faculty of Science and Technology, Department of Mtathematics, Keio University, Hiyoshi 3-14-1, Kohoku, Yokohama, Kanagawa, Japan}
\email{hirakawa@keio.jp}
     \thanks{This research was supported by JSPS KAKENHI Grant Number JP15J05818,
     the Research Grant of Keio Leading-edge Laboratory of Science \& Technology (Grant Numbers 2018-2019 000036 and 2019-2020 000074).
     This research was supported in part by KAKENHI 18H05233.
     This research was conducted as part of the KiPAS program FY2014--2018 of the Faculty of Science and Technology at Keio University as well as the JSPS Core-to-Core program ``Foundation of a Global Research Cooperative Center in Mathematics focused on Number Theory and Geometry".}
\date{\today}
 
\theoremstyle{plain}
 \newtheorem{theorem}{Theorem}[section]
 \crefname{theorem}{Theorem}{Theorems}
 \newtheorem{proposition}[theorem]{Proposition}
 \crefname{proposition}{Proposition}{Propositions}
 \newtheorem{lemma}[theorem]{Lemma}
 \crefname{lemma}{Lemma}{Lemmas}
 
 \crefname{corollary}{Corollary}{Corollaries}
 \newtheorem{conjecture}[theorem]{Conjecture}
 \crefname{conjecture}{Conjecture}{Conjectures}

 \crefname{question}{Question}{Questions}
 
 \crefname{problem}{Problem}{Problems}

\theoremstyle{definition} 
 
 \crefname{definition}{Definition}{Definitions}
 
 \crefname{example}{Example}{Examples}
 
 \crefname{remark}{Remark}{Remarks}

\begin{document}

%%%Title&Author

\maketitle

%%% Abstract 

\begin{abstract}
In this article, we introduce a systematic and uniform construction
of non-singular plane curves of odd degrees $n \geq 5$
which violate the local-global principle.
Our construction works unconditionally 
for $n$ divisible by $p^{2}$ for some odd prime number $p$.
Moreover, our construction also works
for $n$ divisible by some $p \geq 5$
which satisfies a conjecture on $p$-adic properties of the fundamental units of $\mathbb{Q}(p^{1/3})$ and $\mathbb{Q}((2p)^{1/3})$.
This conjecture is a natural cubic analogue of the classical Ankeny-Artin-Chowla-Mordell conjecture
for $\mathbb{Q}(p^{1/2})$
and easily verified numerically.
\end{abstract}

%\tableofcontents

%%%Contents

%%%
%%% Intro
%%%
\section{Introduction}

In the theory of Diophantine equations,
the local-global principle for quadratic forms established by Minkowski and Hasse
is one of the major culminations (cf.\  \cite[Theorem 8, Ch.\  IV]{Serre}).

In contrast,
there exist many homogeneous forms of higher degrees which violate the local-global principle
(i.e., counterexamples to the local-global principle).
For example, Selmer \cite{Selmer} found that a non-singular plane cubic curve defined by
\begin{equation} \label{Selmer}
	3X^{3}+4Y^{3} = 5Z^{3}
\end{equation}
has rational points over $\mathbb{R}$ and $\mathbb{Q}_{p}$ for every prime number $p$
but not over $\mathbb{Q}$.
From \cref{Selmer},
we can easily construct reducible (especially singular) counterexamples of higher degrees.

After that, Fujiwara \cite{Fujiwara} found that a non-singular plane quintic curve defined by
\begin{equation} \label{Fujiwara_example}
	(X^{3}+5Z^{3})(X^{2}+XY+Y^{2}) = 17Z^{5}
\end{equation}
violates the local-global principle.
More recently,
Cohen \cite[Corollary 6.4.11]{Cohen} gave several counterexamples of the form $x^{p}+by^{p}+cz^{p} = 0$
of degree $p = 3, 5, 7, 11$ with $b, c \in \mathbb{Z}$,
and Nguyen \cite{Nguyen_QJM,Nguyen_Tokyo} gave recipes for counterexamples of even degrees and more complicated forms.

In \cite{Poonen-Voloch},
Poonen and Voloch made a qualitative conjecture on an old folklore that
most hypersurfaces of degree $n \geq d+3$ in the projective space $\mathbb{P}^{d+1}$
violate the local-global principle.
Probably based on this folklore and the Poonen-Voloch conjecture,
there are many works
for the existence and the proportion of counterexamples in certain classes
\cite{Bhargava,Bhargava-Gross-Wang,Browning,Clark,Dietmann-Marmon,Granville,Poonen-Voloch}.
Among them,
Dietmann and Marmon \cite[Theorem 2]{Dietmann-Marmon} proved that,
under the abc conjecture \cite{Masser},
among non-singular plane curves of degree $k$
\[
	AX^{k}+BY^{k} = CZ^{k} \quad (A, B, C \in \mathbb{Z} \setminus \{ 0 \})
\]
with rational points over $\mathbb{R}$ and $\mathbb{Q}_{p}$ for every prime number $p$,
100\% of them violate the local-global principle for every $k \in \mathbb{Z}_{\geq 6}$.
However, the argument in \cite[Theorem 2]{Dietmann-Marmon} cannot give any specific (conjectural) counterexamples
because the abc conjecture is ineffective to estimate the candidates of $A, B, C$.
It may be surprising that
there seems to be no other concrete counterexamples
to the local-global principle for non-singular plane curves of odd degrees $\geq 5$ than
\cite{Cohen,Fujiwara,Fujiwara-Sudo}.

In this article, we exhibit how to construct such counterexamples of various odd degrees
in a systematic and uniform manner.
The following \cref{main_odd} is the main theorem of this article.
We should emphasise that
although it is unclear from the statement,
in the proof,
we shall exhibit how to generate parameters in the following equations,
i.e., we can generate as many as we want explicit counterexamples like \cref{Selmer,Fujiwara_example}.

\begin{theorem} \label{main_odd}
Let $p$ be an odd prime number.
Set $P = 2p \ \text{or} \ p$ so that $P \not\equiv \pm 1 \bmod{9}$.
Let $\epsilon = \alpha + \beta p^{1/3} + \gamma p^{2/3} \in \mathbb{R}_{> 1}$
be the fundamental unit of $\mathbb{Q}(P^{1/3})$
with $\alpha, \beta, \gamma \in \mathbb{Z}$.
Set
\[
	\iota =
	\begin{cases}
	1 & \text{if $\beta \not\equiv 0 \pmod{p}$ or $\beta \equiv \gamma \equiv 0 \pmod{p}$} \\
	2 & \text{if $\beta \equiv 0 \pmod{p}$ and $\gamma \not\equiv 0 \pmod{p}$}
	\end{cases}.
\]
Let $n \in \mathbb{Z}_{\geq 5}$ be an odd integer divisible by $p^{\iota}$.
Then, there exist
infinitely many $(n-3)/2$-tuples of pairs of integers $(b_{j}, c_{j})$ ($1 \leq j \leq (n-3)/2$)
%where $(b_{j}, c_{j}) \in \mathbb{Z}^{\oplus 2}$ and $\gcd(b_{j}, c_{j}) = 1$ for every $j$,
satisfying the following condition:

There exist infinitely many $L \in \mathbb{Z}$ such that the equation
\begin{equation} \label{equation}
	(X^{3}+P^{\iota}Y^{3}) \prod_{j = 1}^{\frac{n-3}{2}} (b_{j}^{2}X^{2}+b_{j}c_{j}XY+c_{j}^{2}Y^{2}) = LZ^{n}
\end{equation}
define non-singular plane curves of degree $n$ which violate the local-global principle.

Moreover, for each $n \geq 5$ divisible by $p^{\iota}$,
there exists a set of such $(n-3)/2$-tuples $((b_{j}, c_{j}))_{1 \leq j \leq (n-3)/2}$
which gives infinitely many geometrically non-isomorphic classes of such curves of degree $n$.
\end{theorem}

In particular,
if $\beta \not\equiv 0 \pmod{p}$ for a prime number $p$,
then we can generate an infinite family of explicit counterexamples for every odd degree $n \equiv 0 \pmod{p}$.
The authors conjecture that this hypothesis is always true whenever $p \neq 3$:
\footnote{
Note that \cref{AACM_cubic} holds if and only if
$\mathbb{Q}(P^{1/3})$ has a unit $\alpha_{0}+\beta_{0} P^{1/3}+\gamma_{0} P^{2/3}$
with $\alpha_{0}, \beta_{0}, \gamma_{0} \in (1/3)\mathbb{Z}$
such that $\beta_{0} \not\equiv 0 \pmod{p}$.
The authors verified \cref{AACM_cubic} for all $p < 10^5$ by Magma \cite{Magma}.
%which costs around five hours.
For the detail, see Appendix B.
}

\begin{conjecture} \label{AACM_cubic}
Let $p \neq 3$ be a prime number,
$P = p \ \text{or} \ 2p$,
and $\epsilon = (\alpha+\beta p^{1/3}+\gamma p^{2/3})/3 \in \mathbb{R}_{> 1}$
be the fundamental unit of $\mathbb{Q}(P^{1/3})$
with $\alpha, \beta, \gamma \in \mathbb{Z}$.
Then, we have $\beta \not\equiv 0 \pmod{p}$.
\end{conjecture}

In fact,
\cref{AACM_cubic} is a natural cubic field analogue of the following more classical conjecture
for the real quadratic field $\mathbb{Q}(p^{1/2})$,
\footnote{
For numerical verification of \cref{AACM},
see e.g. \cite{AAC_verification,AAC_verification_2}.
}
whose origin goes back to Ankeny-Artin-Chowla \cite{AAC} for $p \equiv 1 \pmod{4}$
and Mordell \cite{Mordell_Pellian_II} for $p \equiv 3 \pmod{4}$ respectively:

\begin{conjecture} \label{AACM}
Let $p \neq 2$ be a prime number,
and $\epsilon = (\alpha+\beta p^{1/2})/2 \in \mathbb{R}_{> 1}$ be the fundamental unit of $\mathbb{Q}(p^{1/2})$ with $\alpha, \beta \in \mathbb{Z}$.
Then, we have $\beta \not\equiv 0 \pmod{p}$.
\end{conjecture}

A key ingredient of our construction is the following theorem on the distribution of prime numbers
represented by binary cubic polynomials:

\begin{theorem} [{\cite[Theorem 1]{HBM2004}}] \label{HBM}
Let $f_{0} \in \mathbb{Z}[X, Y]$ be an irreducible binary cubic form,
$\rho \in \mathbb{Z}$,
$(\gamma_{1}, \gamma_{2}) \in \mathbb{Z}^{\oplus 2}$,
and $\gamma_{0}$ be
the greatest common divisor of the coefficients of
$f_{0}(\rho x+\gamma_{1}, \rho y+\gamma_{2})$.
Set $f(x, y) := \gamma_{0}^{-1}f_{0}(\rho x+\gamma_{1}, \rho y+\gamma_{2})$.
Suppose that $\gcd(f(\mathbb{Z}^{\oplus 2})) = 1$.
Then, the set $f(\mathbb{Z}^{\oplus 2})$
contains infinitely many prime numbers.
\end{theorem}

In \S2, we give a recipe
which exhibits how to construct counterexamples to the local-global principle as in \cref{equation}
from certain Fermat type equations $x^{3}+P^{\iota}y^{3} = Lz^{n}$ 
and prime numbers of the form $P^{\iota}b^{3}+c^{3}$.
Thanks to \cref{HBM},
both of the above Fermat type equations $x^{3}+P^{\iota}y^{3} = Lz^{n}$ 
and prime numbers of the form $P^{\iota}b^{3}+c^{3}$ are generated in completely explicit manners in \S3 and \S4 respectively.
In \S4, the proof of \cref{main_odd} is given by combining these arithmetic objects
with a geometric argument (\cref{infinitude}) on the non-isomorphy
of complex algebraic curves defined by \cref{equation}.
It should be emphasized that
the infinitude in \cref{main_odd} is a striking advantage of our construction
based on analytic number theory and complex algebraic geometry,
which is contrast to those based on algebraic and computational tools in e.g. \cite{Cohen,Fujiwara,Fujiwara-Sudo}.
In \S5, we demonstrate how our construction works for each given degree
by exhibiting concrete examples of degree $7, 9, 11$.
In Appendix, we explain how we can verify \cref{AACM_cubic} numerically 
for each given prime number $p$ by Magma \cite{Magma}.

It is fair to say that
thanks to \cref{HBM} (and \cref{class_number}),
which is one of the culminations of highly sophisticated modern analytic number theory,
our proof of \cref{main_odd} is relatively elementary
and almost covered by a standard first course of algebraic number theory
(as in e.g. \cite[Part I]{Cohen}, \cite[Ch.\  I -- Ch.\  V]{Hecke}, \cite[Ch.\  I -- Ch.\  V]{Marcus}, and \cite{Samuel}).
Moreover, after we admit \cref{main_odd},
it is quite easy to generate as many as we want explicit counterexamples like \cref{Selmer,Fujiwara_example}.

We would like to conclude this introduction with a comment on the style of our proof.
Each step of our proof
(e.g. \cref{local_solubility,global_unsolubility,Fermat})
can be easily refined to more powerful forms.
\footnote{
For example,
if we make effort for 3-adic and $p$-adic solubility in the proof of \cref{local_solubility},
then we can relax the conditions on the coefficients
and generalize our construction so that it includes the constructions of \cite{Fujiwara,Fujiwara-Sudo}.
Moreover, the use of \cref{HBM} in the  proof of \cref{Fermat} is not essential.
In fact, we can prove that a positive proportion of prime numbers $l$ can be used
to generate the equations in \cref{Fermat}
by using class field theory (cf.\ \cite{Tate_GCFT,Lv-Deng}).
%More precisely, the existence theorem of the class field associate with orders and the Chebotarev density theorem.
}
However, in order to state and prove them in full generality,
we need the several times as much as the present volume.
In order to keep this paper as readable as possible,
we state each proposition in a restricted form
that is sufficient to prove \cref{main_odd}
with complete description of how to generate parameters in \cref{equation}.

%%%%%
%%%%% Hasse
%%%%%
\section{Construction from prime numbers and Fermat type equations}

Let $p$ be an odd prime number, $P = p \ \text{or} \ 2p$, and $\iota = 1 \ \text{or} \ 2$.
Here, we can take $P$ and $\iota$ independently of $p$.
In this section, we prove the following proposition,
which gives explicit counterexamples to the local-global principle of degree $n$
under the assumption that we have
\begin{itemize}
\item
	sufficiently many prime numbers of the form $P^{\iota}b^{3}+c^{3}$ with $b, c \in \mathbb{Z}$ and
	
\item
	integers $L$ such that the equation $x^{3}+P^{\iota}y^{3} = Lz^{n}$ has a specific property.
\end{itemize}
In what follows,
for each prime number $l$,
$v_{l}(n)$ denotes the additive $l$-adic valuation of $n \in \mathbb{Z}$ normalized so that $v_{l}(l) = 1$.

\begin{proposition} \label{recipe_odd}
Let $n \in \mathbb{Z}_{\geq 5}$ be an odd integer,
$p$ be a prime number,
$P = p \ \text{or} \ 2p$,
$\iota = 1 \ \text{or} \ 2$,
and $b_{1}, \dots, b_{\frac{n-3}{2}}, c_{1}, \dots, c_{\frac{n-3}{2}}, L \in \mathbb{Z}$ satisfying the following conditions:
\begin{enumerate}
\item
$P^{\iota}b_{j}^{3}+c_{j}^{3} \equiv 2 \pmod{3}$ is a prime number prime to $P$ for every $j$.

\item
$L = \prod_{l \equiv 2 \pmod{3}} l^{v_{l}(L)}$ with $v_{l}(L) < n$.

\item
$\gcd(L, b_{j}c_{j}) = 1$ for every $j$.

\item
If $P \equiv 0 \pmod{2}$,
then $L \equiv \prod_{j} b_{j}^{2} \equiv 1 \pmod{2}$

\item
If $P \not\equiv \pm1 \pmod{9}$,
then $L \equiv \prod_{j} b_{j}^{2} \not\equiv 0 \pmod{3}$
and $\sum_{j} b_{j}^{-1}c_{j} \not\equiv 0 \pmod{3}$.

\item
If $p \equiv 2 \pmod{3}$,
then $L \equiv \prod_{j} b_{j}^{2} \not\equiv 0 \pmod{p}$
and $\sum_{j} b_{j}^{-1}c_{j} \not\equiv 0 \pmod{p}$.

\item
For every primitive triple $(x, y, z)  \in \mathbb{Z}^{\oplus 3}$ 
\footnote{
In this article,
we say that a triple $(x, y, z)  \in \mathbb{Z}^{\oplus 3}$ is primitive
if $\gcd(x, y, z) = 1$.
}
satisfying $x^{3}+P^{\iota}y^{3} = Lz^{n}$,
there exists a prime divisor $l$ of $L$ such that $x \equiv y \equiv 0 \pmod{l}$.
\end{enumerate}
Then, the equation 
\[
	(X^{3}+P^{\iota}Y^{3})\prod_{j = 1}^{\frac{n-3}{2}}(b_{j}^{2}X^{2}+b_{j}c_{j}XY+c_{j}^{2}Y^{2}) = LZ^{n}
\]
violates the local-global principle.
\end{proposition}

This is a consequence of the following two lemmas.

\begin{lemma} [local solubility] \label{local_solubility}
Let $n \in \mathbb{Z}_{\geq 5}$ be an odd integer,
$p$ be a prime number,
$P = p \ \text{or} \ 2p$,
$\iota = 1 \ \text{or} \ 2$,
and $b_{1}, \dots, b_{\frac{n-3}{2}}, c_{1}, \dots, c_{\frac{n-3}{2}}, L \in \mathbb{Z}$ satisfying the following conditions:
\begin{enumerate}
\item
If $P \equiv 0 \pmod{2}$,
then $L \equiv \prod_{j} b_{j}^{2} \equiv 1 \pmod{2}$

\item
If $P \not\equiv \pm1 \pmod{9}$,
$L \equiv \prod_{j} b_{j}^{2} \not\equiv 0 \pmod{3}$
and $\sum_{j} b_{j}^{-1}c_{j} \not\equiv 0 \pmod{3}$.

\item
If $p \equiv 2 \pmod{3}$,
$L \equiv \prod_{j} b_{j}^{2} \not\equiv 0 \pmod{p}$
and $\sum_{j} b_{j}^{-1}c_{j} \not\equiv 0 \pmod{p}$.
\end{enumerate}
Then, the equation 
\[
	F(X, Y, Z)
	:= (X^{3}+P^{\iota}Y^{3})\prod_{j = 1}^{\frac{n-3}{2}}(b_{j}^{2}X^{2}+b_{j}c_{j}XY+c_{j}^{2}Y^{2}) - LZ^{n}
	= 0
\]
has non-trivial solutions over $\mathbb{R}$ and $\mathbb{Q}_{l}$ for every prime number $l$.
\end{lemma}

\begin{proof}
We prove this statement along Fujiwara's argument in \cite{Fujiwara}:
We may assume that $b_{1}, c_{1} \neq 0$
because if $b_{j} = c_{j} = 0$ for every $j$, there is nothing to prove.
Since $\iota = 1 \ \text{or} \ 2$,
the minimal splitting field of
$(X^{3}+P^{\iota}Y^{3})(b_{1}^{2}X^{2}+b_{1}c_{1}XY+c_{1}^{2}Y^{2})$
is a Galois extension over $\mathbb{Q}$
whose Galois group is isomorphic to the symmetric group of degree 3.
In particular, the residual degree at every prime number is 1, 2, or 3.
Moreover, since this extension is unramified outside 2, 3, and $p$ (and $\infty$),
the polynomial $(X^{3}+P^{\iota}Y^{3})(b_{1}^{2}X^{2}+b_{1}c_{1}XY+c_{1}^{2}Y^{2})$
has a linear factor over $\mathbb{Q}_{l}$ for every prime number $l \neq 2, 3, p$.
Therefore, the assertion follows if one notes the following facts:
\begin{enumerate}
\item
$X^{3}+P^{\iota}Y^{3}$ is decomposed in $\mathbb{R}[X, Y]$.

\item
If $P \equiv 1 \pmod{2}$, then $X^{3}+P^{\iota}Z^{3}$ is decomposed in $\mathbb{Z}_{2}[X, Y]$.
On the other hand, if $P \equiv 0 \pmod{2}$,
then since $F(1, 0, 1) \equiv \prod_{j} b_{j}^{2} - L \equiv 0 \pmod{2}$,
and $(\partial F/\partial Z)(1, 0, 1) = nL \not\equiv 0 \pmod{2}$,
%Here, $n \geq 5$ is implicitly used.
we obtain a 2-adic lift of mod 2 solution $(1, 0, 1)$ by Hensel's lemma.
	
\item
If $P \equiv \pm 1 \pmod{9}$, then $X^{3}+P^{\iota}Z^{3}$ is decomposed in $\mathbb{Z}_{3}[X, Y]$.
On the other hand, if $P \not\equiv \pm 1 \pmod{9}$,
then since $F(1, 0, 1) \equiv \prod_{j} b_{j}^{2} - L \equiv 0 \pmod{3}$,
and $(\partial F/\partial Y)(1, 0, 1) \equiv (\prod_{j} b_{j}^{2}) \cdot (\sum_{j} b_{j}^{-1}c_{j}) \not\equiv 0 \pmod{3}$,
%Here, $n \geq 5$ is implicitly used.
we obtain a 3-adic lift of mod 3 solution $(1, 0, 1)$ by Hensel's lemma.
	
\item
If $p \equiv 1 \pmod{3}$,
then $b_{j}^{2}X^{2}+b_{j}c_{j}XY+c_{j}^{2}Y^{2}$ is decomposed in $\mathbb{Z}_{p}[X, Y]$.
On the other hand, if $p \equiv 2 \pmod{3}$,
then since $F(1, 0, 1) \equiv \prod_{j} b_{j}^{2} - L \equiv 0 \pmod{p}$ and
$(\partial F/\partial Y)(1, 0, 1) \equiv (\prod_{j} b_{j}^{2}) \cdot (\sum_{j} b_{j}^{-1}c_{j}) \not\equiv 0 \pmod{p}$,
we obtain a $p$-adic lift of mod $p$ solution $(1, 0, 1)$ by Hensel's lemma.
\end{enumerate}
This completes the proof.
\end{proof}

\begin{lemma} [global unsolubility] \label{global_unsolubility}
Let $n \in \mathbb{Z}_{\geq 3}$ be an odd integer,
$a, b_{1}, ..., b_{\frac{n-3}{2}}, c_{1}, ..., c_{\frac{n-3}{2}}, L \in \mathbb{Z}$
such that
\begin{enumerate}
\item
$ab_{j}^{3}+c_{j}^{3} \equiv 2 \pmod{3}$ is a prime number prime to $a$ for every $j$,

\item
$L = \prod_{l \equiv 2 \pmod{3}} l^{v_{l}(L)}$ with $v_{l}(L) < n$,

\item
$\gcd(L, b_{j}c_{j}) = 1$ for every $j$, and

\item
for every primitive triple $(x, y, z)  \in \mathbb{Z}^{\oplus 3}$
\footnote{
In this article,
we say that a triple $(x, y, z)  \in \mathbb{Z}^{\oplus 3}$ is primitive
if $\gcd(x, y, z) = 1$.
}
satisfying $x^{3}+ay^{3} = Lz^{n}$,
there exists a prime divisor $l$ of $L$ such that $x \equiv y \equiv 0 \pmod{l}$.
\end{enumerate}
Then, there is no triple $(X, Y, Z) \in \mathbb{Z}^{\oplus 3} \setminus \{ (0, 0, 0) \}$ satisfying
\begin{equation} \label{condition}
	(X^{3}+aY^{3}) \prod_{j = 1}^{\frac{n-3}{2}} (b_{j}^{2}X^{2}+b_{j}c_{j}XY+c_{j}^{2}Y^{2}) = LZ^{n}.
\end{equation}
\end{lemma}

\begin{proof}
We prove the assertion by contradiction.
Let $(X, Y, Z) \in \mathbb{Z}^{\oplus 3}$ be a triple satisfying \cref{condition}.
We may assume that it is primitive.
It is sufficient to deduce that
\begin{equation} \label{gcd_1}
	\gcd((X^{3}+aY^{3})L, b_{j}^{2}X^{2}+b_{j}c_{j}XY+c_{j}^{2}Y^{2}) = 1
	\quad \text{for every $j$}.
\end{equation}
Indeed, if \cref{gcd_1} holds,
we have some divisor $z$ of $Z$ satisfying $X^{3}+aY^{3}= Lz^{n}$.
Hence, the fourth assumption implies that
$X \equiv Y \equiv 0 \pmod{l}$ for some prime divisor $l$ of $L$.
However, since $v_{l}(L) < n$,
we have $Z \equiv 0 \pmod{l}$, which contradicts $\gcd(X, Y, Z) = 1$.
In what follows, we deduce \cref{gcd_1}.

First, suppose that
a prime divisor $q$ of $X^{3}+aY^{3}$ divides $b_{j}^{2}X^{2}+b_{j}c_{j}XY+c_{j}^{2}Y^{2}$ for some $j$.
Then, $q$ also divides
\[
	b_{j}^{3}(X^{3}+aY^{3}) - (b_{j}X-c_{j}Y)(b_{j}^{2}X^{2}+b_{j}c_{j}XY+c_{j}^{2}Y^{2})
	= (ab_{j}^{3}+c_{j}^{3})Y^{3}.
\]
Since $\gcd(X, Y, Z) = 1$ and $v_{q}(L) < n$,
we see that $Y \not\equiv 0 \pmod{q}$.
Hence, by the first assumption, we have $q = ab_{j}^{3}+c_{j}^{3} \equiv 2 \pmod{3}$.
In particular, the polynomial $b_{j}^{2}T^{2}+b_{j}c_{j}T+c_{j}^{2}$ is irreducible in $\mathbb{Z}_{q}[T]$.
Since $b_{j}^{2}X^{2}+b_{j}c_{j}XY+c_{j}^{2}Y^{2} \equiv 0 \pmod{q}$
and $Y \not\equiv 0 \pmod{q}$,
we have $c_{j} \equiv 0 \pmod{q}$.
However, $q = ab_{j}^{3}+c_{j}^{3}$ implies that $a$ must be divisible by $q$, a contradiction.

Secondly, suppose that a prime divisor $l \equiv 2 \pmod{3}$ of $L$ divides
$b_{j}^{2}X^{2}+b_{j}c_{j}XY+c_{j}^{2}Y^{2}$ for some $j$.
Then, since $T^{2}+T+1$ is irreducible in $\mathbb{F}_{l}[T]$,
we have $b_{j}X \equiv c_{j}Y \equiv 0 \pmod{l}$.
On the other hand,
since $\gcd(X, Y, Z) = 1$ and $v_{l}(L) < n$,
we see that $X \not\equiv 0 \pmod{l}$ or $Y \not\equiv 0 \pmod{l}$.
However, if $X \not\equiv 0 \pmod{l}$ (resp. $Y \not\equiv 0 \pmod{l}$),
then $b_{j} \equiv 0 \pmod{l}$ (resp. $c_{j} \equiv 0 \pmod{l}$),
which contradicts that $\gcd(L, b_{j}c_{j}) = 1$.

This completes the proof.
\end{proof}

%%%%%
%%%%% Fermat type
%%%%%
\section{Fermat type equations of the form $X^{3}+p^{\iota}Y^{3} = LZ^{n}$}

In this section,
we take an odd prime number $p$ and set $P = 2p \ \text{or} \ p$ so that $P \not\equiv \pm 1 \bmod{9}$.
We fix them through the whole of this section.
Let $\pi = P^{1/3} \in \mathbb{R}$ be the real cubic root of $P$,
$K = \mathbb{Q}(\pi) \subset \mathbb{R}$,
and $\mathcal{O}_{K}$ denotes the ring of integers in $K$.
Here, note that since $P \not\equiv \pm1 \pmod{9}$,
we see that $\mathcal{O}_{K} = \mathbb{Z}[\pi]$.
Let $\epsilon = \alpha+\beta \pi+\gamma \pi^{2} > 1$ be the fundamental unit of $K$
with $\alpha, \beta, \gamma \in \mathbb{Z}$,
and $\mathfrak{p} = \mathfrak{p}_{p}$ be the unique prime ideal of $K$ above $p$.
Then, we see that $\mathfrak{p}^{3} = p\mathcal{O}_{K}$
and $\pi$ is a uniformizer of the $\mathfrak{p}$-adic completion of $\mathcal{O}_{K}$.
Note that the Galois closure of $K$ in $\mathbb{C}$ is $K(\zeta_{3})$,
where $\zeta_{3} \in \mathbb{C}$ is a fixed primitive cubic root of unity.
For basic properties of these objects, see \cite{Dedekind,Barrucand-Cohn} and their references.

Set
\[
	\iota =
	\begin{cases}
	1 & \text{if $\beta \not\equiv 0 \pmod{p}$ or $\beta \equiv \gamma \equiv 0 \pmod{p}$} \\
	2 & \text{if $\beta \equiv 0 \pmod{p}$ and $\gamma \not\equiv 0 \pmod{p}$}
	\end{cases}.
\]
For example,
if $P = 3 \ \text{or} \ 6$, then we have $(\alpha, \beta, \gamma) = (4, 3, 2) \ \text{or} \ (109, 60, 33)$,
hence $\iota = 2 \ \text{or} \ 1$ respectively.
On the other hand,
if \cref{AACM_cubic} holds for $p \geq 5$,
then we have $\iota = 1$ for $P = p \ \text{and} \ 2p$.

In this section, we prove the following theorem.

\begin{theorem} \label{Fermat}
Let $p \geq 3$ be a prime number,
and $n \in \mathbb{Z}_{\geq 5}$ divisible by $p^{\iota}$.
Then, there exist infinitely many prime numbers $l$
and an integer $m \in \{ 1, 2, ..., p-1 \}$ such that
\begin{enumerate}
\item
$l \equiv 2 \pmod{3}$ and $l^{m} \equiv 1 \pmod{3}$,

\item
$l^{m} \equiv 1 \pmod{p}$, and

\item
every primitive solution of $x^{3}+P^{\iota}y^{3} = l^{m}z^{n}$ satisfies $x \equiv y \equiv 0 \pmod{l}$.
\end{enumerate}
\end{theorem}

In order to prove \cref{Fermat},
we use \cref{HBM}.
Suppose that $p \neq 3$.
Let $h(A, C) = (3P^{\iota}A+1)^{3}+P^{2\iota}(3P^{\iota}C+1)^{3}$.
Then, since $\gcd(h(0, 0), h(1, 0), h(-1, 0)) = 1$,
%$gcd(f) = (1+P^{2\iota}, 27P^{3\iota}+28P^{2\iota}+9P^{\iota}+1, -27P^{3\iota}+28P^{2\iota}-9P^{\iota}+1) = (1+P^{2\iota}, -18P^{\iota}-27, 18P^{\iota}-27) = (4+4P^{2\iota}, 2P^{\iota}+3, 2P^{\iota}-3) = (13, 2P^{\iota}+3, 6) = 1$
we have $\gcd(h(\mathbb{Z}^{\oplus 2})) = 1$.
Therefore, \cref{HBM} implies that
there exist infinitely many prime numbers $l$ of the form
\[
	l = a^{3}+P^{2\iota}c^{3}
	\quad \text{with} \quad (a, c) = (3P^{\iota}A+1, 3P^{\iota}C+1) \in \mathbb{Z}^{\oplus 2}.
\]
On the other hand,
if $p = 3$,
we can use
$h(A, C) = (PA-1)^3+P^{2\iota}(PC+1)^{3}$.
Thus, if one notes that $2^{m} \equiv 1 \pmod{3}$ if and only if $m$ is even,
the proof of \cref{Fermat} is reduced to prove the following proposition.

\begin{proposition} \label{heuristic_recipe}
Let $p$ be an odd prime number
and $l \equiv 2 \pmod{3}$ be a prime number prime to $P$.
Suppose that there exist $a, b, c \in \mathbb{Z}$ satisfying the following conditions:
\begin{enumerate}
	\item
	$l = a^{3}+P^{\iota}b^{3}+P^{2\iota}c^{3}-3P^{\iota}abc$.
	
	\item
	$a \equiv \pm 1 \pmod{p}$, $b \equiv 0 \pmod{p}$, and $c \not\equiv 0 \pmod{p}$.
	
	\item
	If $p = 5$, then additionally $c \not\equiv -a \pmod{5}$.
	
	\item
	If $P = 3$, then additionally $c \equiv -a \pmod{3}$.
\end{enumerate}
Then, there exists a positive even integer $m < p$ such that
every primitive solution of $x^{3}+P^{\iota}y^{3} = l^{m}z^{p^{\iota}}$ satisfies $x \equiv y \equiv 0 \pmod{l}$.
\end{proposition}

First, we prove the following proposition as an intermediate step.

%%%
%%% Fujiwara
%%%
\begin{proposition} \label{Fujiwara}
Let $p$ be a prime number,
$l$ be a prime number such that $l$ is prime to $P$ and $l \equiv 2 \pmod{3}$,
and $m \in \mathbb{Z}_{\geq 1}$.
Assume that there exist $a+b\pi^{\iota}+c\pi^{2\iota} \in \mathcal{O}_{K}$
with $a, b, c \in \mathbb{Z}$ satisfying the following conditions:
\begin{enumerate}
\item
	$l = a^{3}+b^{3}P^{\iota}+c^{3}P^{2\iota}-3abcP^{\iota}$.
	\footnote{
	Since $l \equiv 2 \pmod{3}$,
	$\mathcal{O}_{K}$ has prime ideals $\mathfrak{p}_{l}$ and $\mathfrak{p}_{l^{2}}$
	of norms of degree 1 and 2 respectively.
	Therefore, the first condition holds (up to signature)
	if and only if $\mathfrak{p}_{l}$ is generated by $a+b\pi+c\pi^{2}$.
	}
	
\item
	If we define $A_{k}, B_{k}, C_{k} \in \mathbb{Z}$ by
	\[
		A_{k}+B_{k}\pi^{\iota}+C_{k}\pi^{2\iota}
		= \epsilon^{k}(a+b\pi^{\iota}+c\pi^{2\iota})^{m} \quad (k \in \mathbb{Z}),
	\]
	then we have $C_{k} \not\equiv 0 \pmod{p}$ for every $k \in \mathbb{Z}$.
\end{enumerate}
Then, for every $n \in \mathbb{Z}_{\geq 5}$ divisible by $p^{\iota}$,
every primitive solution of $x^{3}+P^{\iota}y^{3} = l^{m}z^{n}$
satisfies $x \equiv y \equiv 0 \pmod{l}$.
\end{proposition}

We prove \cref{Fujiwara} along a classical idea as done in \cite{Fujiwara},
where Fujiwara proved the insolubility of $x^{3}+5y^{3} = 17z^{5}$.
We use the following lemma,
which is an immediate consequence of
\cite[Corollarie 2]{Barrucand-Louboutin}, \cite[Theorem 3]{Cusick}, and \cite[Corollary 4.2.1]{Barrucand-Cohn}.
\footnote{
Indeed, we may check \cref{class_number} (1) for $p \leq 139$ directly
(e.g. by Magma \cite{Magma}).
If $p > 139$,
then
by \cite[Corollarie 2]{Barrucand-Louboutin} and \cite[Theorem 3]{Cusick},
the class number of $K$ is bounded by
\[
	\frac{1}{4}p \cdot \frac{2\log p + \log 3}{2\log p - 2\log 3}
		\quad \text{or} \quad
			\frac{3}{4}p \cdot \frac{2\log p + 3\log 3}{2\log p}
\]
according to $p \equiv \pm1 \pmod{9}$ or not.
For \cref{class_number} (2),
the similar argument works if $p \equiv \pm4 \bmod{4}$.
If $p \not\equiv \pm4 \bmod{9}$,
we use \cite[Corollary 4.2.1]{Barrucand-Cohn} to deduce that
the class number of $K$ is divisible by $3$
and the quotient is strictly smaller than $p$ by the above argument.
}

\begin{lemma} \label{class_number}
Let $p$ be a prime number.
\begin{enumerate}
\item
The class number of $K = \mathbb{Q}(p^{1/3})$ is smaller than $p$.

\item
The class number of $K = \mathbb{Q}((2p)^{1/3})$ is prime to $p$.
\end{enumerate}
\end{lemma}

%%%
%%% Fujiwara's construction
%%%
\begin{proof} [Proof of \cref{Fujiwara}]
We prove the assertion by contradiction.
Suppose that there exists a primitive triple $(x, y, z) \in \mathbb{Z}^{\oplus 3}$
such that $x^{3}+p^{\iota}y^{3} = l^{m}z^{n}$, and either $x$ or $y$ is prime to $l$.

First, note that
since either $x$ or $y$ is prime to $l$ and $\gcd(l, P) = 1$,
$x^{2}-xy\pi^{\iota}+y^{2}\pi^{2\iota}$ cannot be divisible by $l$.
Moreover,
$l \equiv 2 \pmod{3}$ splits to the product of two prime ideals
$\mathfrak{p}_{l}$ and $\mathfrak{p}_{l^{2}} $ of degree 1 and 2 respectively.
Suppose that $x+y\pi^{\iota}$ is divisible by $\mathfrak{p}_{l^{2}}$.
Then, the product of its conjugates
$(x+\zeta_{3} y\pi^{\iota})(x+\zeta_{3}^{2} y\pi^{\iota})
= x^{2}-xy\pi^{\iota}+y^{2}\pi^{2\iota}$
is divisible by $l$, a contradiction
(cf.\  the following argument for $q \equiv 2 \pmod{3}$).
Therefore, $x^{2}-xy\pi^{\iota}+y^{2}\pi^{2\iota}$ is divisible by $\mathfrak{p}_{l^{2}}^{m}$
but not divisible by $\mathfrak{p}_{l}$.
Accordingly,
$x+y\pi^{\iota}$ is divisible by $\mathfrak{p}_{l}^{m}$ but not divisible by $\mathfrak{p}_{l^{2}}$.
%Up to here, we do not use the condition that $\mathfrak{p}_{l}$ is principal.

Next, suppose that $x+y\pi^{\iota}$ is divisible by a prime ideal above a prime divisor $q$ of $z$.
Then, note that if $\gcd(q, P) = 1$ and $x+y\pi^{\iota}$ or $x^{2}-xy\pi^{\iota}+y^{2}\pi^{2\iota}$ is divisible by $q$ itself,
%If $P \equiv \pm 1 \pmod{9}$, then we should assume that $\gcd(q, 3P)$, which is a crucial point.
then we have $x \equiv y \equiv 0 \pmod{q}$, which contradicts that $(x, y, z)$ is primitive.
On the other hand, since $P \not\equiv \pm1 \pmod{9}$,
the possible decomposition types of $q$ in $K$ are as follows:
\begin{enumerate}
\item
	$(q) = \mathfrak{p}_{q, 1} \mathfrak{p}_{q, 2} \mathfrak{p}_{q, 3}$,
	i.e., $q \equiv 1 \pmod{3}$ and $P \pmod{q} \in \mathbb{F}_{q}^{\times 3}$
	
\item
	$(q) = \mathfrak{p}_{q} \mathfrak{p}_{q^{2}}$, i.e., $q \equiv 2 \pmod{3}$ and $P \not\equiv 0 \pmod{q}$

\item
	$(q) = \mathfrak{p}_{q}^{3}$, i.e., $P \equiv 0 \pmod{q}$ or $q = 3$.
\end{enumerate}
In each case, we have the following conclusion:
\begin{enumerate}
\item
	If $x+y\pi^{\iota}$ is divisible by
	distinct two prime ideals above $q$, say $\mathfrak{p}_{q, 1}$ and $\mathfrak{p}_{q, 2}$, 
	then $x^{2}-xy\pi^{\iota}+y^{2}\pi^{2\iota}$
	is divisible by $(\mathfrak{p}_{q, 1}\mathfrak{p}_{q, 3}) \cdot (\mathfrak{p}_{q, 2}\mathfrak{p}_{q, 3})$,
	hence by $q$, a contradiction.
	Therefore, $x+y\pi^{\iota}$ is divisible by $\mathfrak{p}_{q, 1}^{nv_{q}(z)}$
	but not by $\mathfrak{p}_{q, 2}$ nor $\mathfrak{p}_{q, 3}$
	if we replace $\mathfrak{p}_{q, 1}, \mathfrak{p}_{q, 2}, \mathfrak{p}_{q, 3}$ to each other if necessary.
	
\item
	In this case,
	$q$ is decomposed in $K(\zeta_{3})$ so that
	$\mathfrak{p}_{q} = \mathfrak{P}_{q^{2}, 1}$ and
	$\mathfrak{p}_{q^{2}} = \mathfrak{P}_{q^{2}, 2}\mathfrak{P}_{q^{2}, 3}$.
	If $x+y\pi^{\iota}$ is divisible by $\mathfrak{p}_{q^{2}}$,
	then	$x^{2}-xy\pi^{\iota}+y^{2}\pi^{2\iota}$
	is divisible by
	$(\mathfrak{P}_{q^{2}, 1}\mathfrak{P}_{q^{2}, 2}) \cdot (\mathfrak{P}_{q^{2}, 1}\mathfrak{P}_{q^{2}, 3})$,
	hence by $q$, a contradiction.
	Therefore, $x+y\pi^{\iota}$ is divisible by $\mathfrak{p}_{q}^{nv_{q}(z)}$
	but not by $\mathfrak{p}_{q^{2}}$.

\item
	In this case, since $x^{3}+p^{\iota}y^{3}$ is divisible by $\mathfrak{p}_{q}^{3n}$,
	$x+y\pi^{\iota}$ is divisible by $q$.
	%$n \geq 3$ is implicitly used.
	This implies that both $x$ and $y$ are divisible by $q$, which contradicts that $(x, y, z)$ is primitive.
\end{enumerate}

As a consequence, we see that
there exists an integral ideal $\mathfrak{w}$ of $\mathcal{O}_{K}$ such that
\[
	(x+y\pi^{\iota})
	= \mathfrak{p}_{l}^{m}\mathfrak{w}^{n} \quad
	\text{and} \quad
	(P, \mathfrak{w}) = 1.
\]
Since the first assumption implies that $\mathfrak{p}_{l}$ is generated by $a+b\pi^{\iota}+c\pi^{2\iota}$,
$\mathfrak{w}^{n}$ is also a principal ideal.
Moreover, Lemma \ref{class_number} implies that
$\mathfrak{w}^{n/p^{\iota}}$ is also generated by a single element
$w_{0}+w_{1}\pi+w_{2}\pi^{2} \in \mathcal{O}_{K}$ with $w_{0}, w_{1}, w_{2} \in \mathbb{Z}$.
Therefore, there exists $k \in \mathbb{Z}$ such that
%In fact, we can take $k \in \{ 0, 1, ..., p^{\iota}-1 \}$.
\begin{align*}
	x+y\pi^{\iota}
	&=
	\epsilon^{k}(a+b\pi^{\iota}+c\pi^{2\iota})^{m}(w_{0}+w_{1}\pi+w_{2}\pi^{2})^{p^{\iota}} \\
	&\equiv
	A_{k}w_{0}^{n}+B_{k}w_{0}^{n}\pi^{\iota}+C_{k}w_{0}^{n}\pi^{2\iota} \pmod{\mathfrak{p}^{2\iota+1}}
\end{align*}
In particular, we have $C_{k}w_{0} \equiv 0 \pmod{p}$.
On the other hand,
since $(P, \mathfrak{w}) = 1$, we have $w_{0} \not\equiv 0 \pmod{p}$.
Therefore, $C_{k} \equiv 0 \pmod{p}$ for some $k$,
which contradicts the assumption.
This completes the proof.
\end{proof}

Now, we can prove \cref{heuristic_recipe}.
Let $\rho(X, Y, Z) := \frac{Y}{2X}-\frac{Z}{Y} \in \mathbb{Q}(X, Y, Z)$
and
\[
	\delta(X, Z) :=
	\begin{cases}
	\rho(\alpha, \beta, \gamma)^{2}-2 \cdot \frac{Z}{X} & \text{if $\beta \not\equiv 0 \pmod{p}$} \\
	\rho(\alpha, \gamma, \frac{\beta}{p})^{2}-2 \cdot \frac{Z}{X} & \text{if $\beta \equiv 0 \pmod{p}$ and $\gamma \not\equiv 0 \pmod{p}$}
	\end{cases}	
	\in \mathbb{Q}(X, Z).
\]
%Here and after we again use $p \geq 3$.

\begin{lemma} \label{quadratic_reduction}
Let $a, c \in \mathbb{Z}$.
Let $(A_{k}, B_{k}, C_{k}) \in \mathbb{Z}^{\oplus 3}$ ($k \in \mathbb{Z}$) such that
\[
	A_{k}+B_{k}\pi^{\iota}+C_{k}\pi^{2\iota}
	\equiv \epsilon^{k}(a+c\pi^{2\iota}) \pmod{\pi^{3 \iota}}.
\]
\begin{enumerate}
\item
Suppose that $\beta \not\equiv 0 \pmod{p}$,
or $\beta \equiv 0 \pmod{p}$ and $\gamma \not\equiv 0 \pmod{p}$.
Then, $C_{k} \not\equiv 0 \pmod{p}$ for every $k$
if and only if $\delta(a, c)$ is not a quadratic residue modulo $p$.

\item
Suppose that $\beta \equiv \gamma \equiv 0 \pmod{p}$.
Then, $C_{k} \not\equiv 0 \pmod{p}$ for every $k$
if and only if $c \not\equiv 0 \pmod{p}$.
\end{enumerate}
\end{lemma}

\begin{proof}
First, by a simple induction on $k \in \mathbb{Z}$, we have
\[
	\frac{\epsilon^{k}}{\alpha^{k}}
	\equiv \begin{cases}
	1+k\cdot\frac{\beta}{\alpha}\pi
		+ k \cdot \left( \frac{k-1}{2}\cdot\frac{\beta^{2}}{\alpha^{2}}
			+ \frac{\gamma}{\alpha} \right) \pi^{2}
				\pmod{\pi^{3}}
					& \text{if $\iota = 1$} \\
	1+k\cdot\frac{\gamma}{\alpha}\pi^{2}
		+ k \cdot \left( \frac{k-1}{2}\cdot\frac{\gamma^{2}}{\alpha^{2}}
			+ \frac{\frac{\beta}{p}}{\alpha} \right) \pi^{4}
				\pmod{\pi^{6}}
					& \text{if $\iota = 2$}
	\end{cases}
\]
for every $k \in \mathbb{Z}$.
This implies that
\[
	\frac{C_{k}}{\alpha^{k}}
	\equiv \begin{cases}
	\frac{\beta^{2}}{2\alpha^{2}}ak^{2}
		- \left( \frac{\beta^{2}}{2\alpha^{2}} - \frac{\gamma}{\alpha} \right)ak
			+ c \pmod{p}
				& \text{if $\iota = 1$}	 \\
	\frac{\gamma^{2}}{2\alpha^{2}}ak^{2}
		- \left( \frac{\gamma^{2}}{2\alpha^{2}} - \frac{\frac{\beta}{p}}{\alpha} \right)ak
			+ c \pmod{p}
				& \text{if $\iota = 2$}	 \\
	\end{cases}
\]
for every $k \in \mathbb{Z}$.
Therefore, $C_{k} \not\equiv 0 \pmod{p}$ for every $k$
if and only if the above quadratic polynomial of $k$ has no zeros in $\mathbb{F}_{p}$.
This implies the assertion.
\end{proof}
%For (1) and (2), the discriminant of the polynomial of k coincides with $\alpha^{-2}\beta^{2}\delta(a, c)$.
%For (3), the polynomial of k becomes a constant $c$.

\begin{proof} [Proof of \cref{heuristic_recipe}]
First, note that \cref{quadratic_reduction} shows that
the assertion holds if $\beta \equiv \gamma \equiv 0 \pmod{p}$.

Suppose that $\beta \not\equiv 0 \pmod{p}$,
or $\beta \equiv 0 \pmod{p}$ and $\gamma \not\equiv 0 \pmod{p}$.
For every $m \in \mathbb{Z}$, define $(a^{(m)}, b^{(m)}, c^{(m)}) \in \mathbb{Z}^{\oplus 3}$ by
\[
	a^{(m)}+b^{(m)}\pi^{\iota}+c^{(m)}\pi^{2\iota}
		= (a+c\pi^{2\iota})^{m}.
\]
Then, by the assumption,
we have
$(a, b, c) \equiv (\pm1+pA, pB, c) \pmod{p^{2}}$ with some $A, B \in \mathbb{Z}$,
hence $(a^{(m)}, b^{(m)}, c^{(m)}) \equiv ((\pm 1)^{m}, 0, \pm c) \pmod{p}$.
Therefore, for every positive even integer $m < p$,
we have
\[
	\delta(a^{(m)}, c^{(m)}) \equiv
	\begin{cases}
	\rho(\alpha, \beta, \gamma)^{2} \mp 2cm \pmod{p}
		& \text{if $\beta \not\equiv 0 \pmod{p}$} \\
	\rho(\alpha, \gamma, p^{-1}\beta)^{2} \mp 2cm \pmod{p}
		& \text{if $\beta \equiv 0 \pmod{p}$ and $\gamma \not\equiv 0 \pmod{p}$}
	\end{cases}.
\]

Suppose that $\beta \not\equiv 0 \pmod{p}$.
Then, by taking into account of \cref{Fujiwara} and \cref{quadratic_reduction},
what we have to prove is that
\[
	\left\{ \rho(\alpha, \beta, \gamma)^{2}-4cm' \bmod{p} \setmid 1 \leq m' \leq \frac{p-1}{2} \right\}
\]
contains a quadratic non-residue.
In particular, it is sufficient to prove that
\[
	B := \left\{ -\frac{1}{4c}\rho(\alpha, \beta, \gamma)^{2}+m' \bmod{p} \setmid 1 \leq m' \leq \frac{p-1}{2} \right\}
\]
contains both quadratic residues and non-residues,
i.e.,
\[
	\left| \sum_{n \in B} \left( \frac{n}{p} \right) \right| < \frac{p-3}{2},
\]
where $(n/p)$ is the quadratic residue symbol.
If $p = 19, 23, 31$ or $p \geq 37$,
then the above inequality follows from an explicit version of the P\'olya-Vinogradov inequality,
e.g.\ the following one due to Pomerance \cite{Pomerance}:
\[
	\left|\sum_{M \leq n \leq N} \left( \frac{n}{p} \right) \right|
	\leq \begin{cases}
	\displaystyle \left( \frac{2}{\pi^{2}}\log p + \frac{4}{\pi^{2}}\log\log p + \frac{3}{2}\right)p^{1/2} & \text{if $p \equiv 1 \pmod{4}$} \\
	\displaystyle \left(\frac{1}{2\pi}\log p + \frac{1}{\pi}\log\log p + 1 \right)p^{1/2} & \text{if $p \equiv 3 \pmod{4}$},
	\end{cases}
\]
For $p = 11, 13, 17, 29$,
we can check that $B$ contains both quadratic residues and non-residues.
In fact, we have the following table,
which implies the desired inequality.
Here, note that $(-n/p) \equiv (-1)^{(p-1)/2}(n/p)$.
\begin{table}[h]
	\begin{flushleft}
	\begin{tabular}{|c|c|c|c|c|c|c|c|c|c|c|c|} \hline
		$n$ & $0$ & $1$ & $2$ & $3$ & $4$ & $5$ & $6$ & $7$ & $8$ & $9$ & $10$ \\ \hline
		$\left( \frac{n}{11} \right)$ & $0$ & $1$ & $-1$ & $1$ & $1$ & $1$ & $-1$ & $-1$ & $-1$ & $1$ & $-1$ \\ \hline
	\end{tabular}
	\end{flushleft}
	
	\begin{flushleft}
	\begin{tabular}{|c|c|c|c|c|c|c|c|c|c|c|c|c|c|} \hline
		$n$ & $0$ & $1$ & $2$ & $3$ & $4$ & $5$ & $6$ & $7$ & $8$ & $9$ & $10$ & $11$ & $12$ \\ \hline
		$\left( \frac{n}{13} \right)$ & $0$ & $1$ & $-1$ & $-1$ & $1$ & $-1$ & $1$ & $1$ & $-1$ & $1$ & $-1$ & $-1$ & $1$ \\ \hline
	\end{tabular}
	\end{flushleft}
	
	\begin{flushleft}
	\begin{tabular}{|c|c|c|c|c|c|c|c|c|c|c|c|c|} \hline
		$n$ & $0$ & $1$ & $2$ & $3$ & $4$ & $5$ & $6$ & $7$ & $8$ & $9$ & $\dots$ & $16$ \\ \hline
		$\left( \frac{n}{17} \right)$ & $0$ & $1$ & $1$ & $-1$ & $1$ & $-1$ & $-1$ & $-1$ & $1$ & $1$ & $\dots$ & $1$ \\ \hline
	\end{tabular}
	\end{flushleft}
	
	\begin{flushleft}
	\begin{tabular}{|c|c|c|c|c|c|c|c|c|c|c|c|c|c|c|c|c|c|c|} \hline
		$n$ & $0$ & $1$ & $2$ & $3$ & $4$ & $5$ & $6$ & $7$ & $8$ & $9$ & $10$ & $11$ & $12$ & $13$ & $14$ & $15$ & $\dots$ & $29$ \\ \hline
		$\left( \frac{n}{29} \right)$ & $0$ & $1$ & $-1$ & $-1$ & $1$ & $1$ & $1$ & $1$ & $-1$ & $1$ & $-1$ & $-1$ & $-1$ & $1$ & $-1$ & $-1$ & $\dots$ & $1$ \\ \hline
	\end{tabular}
	\end{flushleft}
\end{table}

If $P = 7$, then we have $\epsilon = 4+2 \cdot 7^{1/3}+7^{2/3}$,
so $\beta \not\equiv 0 \pmod{7}$ and $\rho(\alpha, \beta, \gamma) \equiv 1 \pmod{7}$.
Therefore, we have $\delta(a^{(m)}, c^{(m)}) \equiv 1 - 2acm \pmod{7}$.
Suppose that $a \equiv 1 \pmod{7}$.
In this case, if $c \equiv 1, 2, 3, 4, 5, 6 \pmod{7}$,
then we can take $m = 6, 4, 2, 2, 4, 2$ respectively.
Suppose that $a \equiv -1 \pmod{7}$.
In this case, if $c \equiv 1, 2, 3, 4, 5, 6 \pmod{7}$,
then we can take $m = 2, 4, 2, 2, 4, 6$ respectively
according to the case of $a \equiv 1 \pmod{7}$.

If $P = 14$, then we have $\epsilon = 1+2 \cdot 7^{1/3}-7^{2/3}$,
so $\beta \not\equiv 0 \pmod{7}$ and $\rho(\alpha, \beta, \gamma) \equiv 5 \pmod{7}$.
Therefore, we have $\delta(a^{(m)}, c^{(m)}) \equiv 4 - 2acm \pmod{7}$.
Suppose that $a \equiv 1 \pmod{7}$.
In this case, if $c \equiv 1, 2, 3, 4, 5, 6 \pmod{7}$,
then we can take $m = 4, 2, 2, 6, 2, 4$ respectively.
The case of $a \equiv -1 \pmod{7}$ is similar.
%Suppose that $a \equiv -1 \pmod{7}$.
%In this case, if $c \equiv 1, 2, 3, 4, 5, 6 \pmod{7}$,
%then we can take $m = 4, 2, 6, 2, 2, 4$ respectively
%according to the case of $a \equiv 1 \pmod{7}$.

If $P = 5$, then we have
%$P = 10$ is not allowed.
$\epsilon = 41+24 \cdot 5^{1/3}+14 \cdot 5^{2/3}$,
so $\beta \not\equiv 0 \pmod{5}$ and $\rho(\alpha, \beta, \gamma) \equiv 1 \pmod{5}$.
Suppose that $a \equiv 1 \pmod{7}$.
In this case, if $c \equiv 1, 2, 3 \pmod{5}$, then we can take $m = 2, 2, 4$.
(Note, however, that if $c \equiv 4 \pmod{5}$, then we cannot take even $m$.)
The case of $a \equiv -1 \pmod{5}$ is similar.

If $P = 6$, then since $\beta \equiv \gamma \equiv 0 \pmod{3}$
we obtain the conclusion.
%$(a^{(1)}, c^{(1)}) \equiv (-1, c) \pmod{3}$
If $P = 3$, then by assumption, we have $a \equiv 2 \pmod{3}$,
hence $(a^{(2)}, c^{(2)}) \equiv (1, c) \pmod{3}$.
%$(a^{(1)}, c^{(1)}) \equiv (-1, c) \pmod{3}$
Since $\epsilon = 2-3^{2/3}$, we have $\beta = 0$ and $\rho(\alpha, \gamma, 3^{-1}\beta) \equiv 2 \pmod{3}$.
Therefore, we have
\[
	\delta(a^{(2)}, c^{(2)}) \equiv 1-2c \pmod{3},
\]
which is quadratic non-residue if and only if $c \equiv 1 \pmod{3}$.
The case of $a \equiv -1 \pmod{3}$ is similar.
This completes the proof in the case of $\beta \not\equiv 0 \pmod{p}$.
The case where $\beta \equiv 0 \pmod{p}$ and $\gamma \not\equiv 0 \pmod{p}$ is similar.\end{proof}

%%%
%%% proof of main theorem
%%%
\section{Proof of the main theorem}

\begin{proof} [Proof of \cref{main_odd}]
First, suppose that $p \neq 3$ and $\iota = 1$.
Let $f(X, Y) = P(3PX\pm1)^{3} + (3PY+1)^{3}$ and $g(X, Y) = P(3PX\mp1)^{3} + (3PY+3)^{3}$
according to $P \equiv \pm1 \pmod{3}$.
\footnote{
	In fact, we can use more general polynomials to generate $b_{j}, c_{j}$ with small absolute values.
	For example, if $P = p \equiv 1 \bmod{3}$,
	then we can use $f(X, Y) = P(3X+1)^{3} + (3Y+1)^{3}$ and $g(X, Y) = P(3X-1)^{3} + (3Y)^{3}$.
	Here, we take $f$ and $g$ so that the proof of \cref{main_odd} gets simple.
	}
Then, since $\gcd(f(0, 0), f(0, 1), f(0, -1)) = \gcd(g(0, 0), g(0, 1), g(0, -1)) = 1$,
we have $\gcd(f(x, y) \mid (x, y) \in \mathbb{Z}^{\oplus 2}) = \gcd(g(x, y) \mid (x, y) \in \mathbb{Z}^{\oplus 2}) = 1$.
%If P \equiv 1 \pmod{3}, then $gcd(f) =(P+1, P+(3P+1)^3) = \gcd(P+1, 3P(9P^2+9P+3)) = \gcd(P+1, 3P^2+3P+1) = \gcd(P+1, 1) = 1$ and $gcd(g) = \gcd(-P+27, -P+(3P+3)^3, -P+(-3P+3)^3) = \gcd(P-27, 27*(P^2+3P+3), 27*(-P^2+3P-3)) = \gcd(P-27, P^2+3P+3, 6P) = \gcd(P-27, P^2+1P+1, 2) = 1$.
%If P \equiv -1 \pmod{3},hen $gcd(f) =(-P+1, -P+(-3P+1)^3) = \gcd(P-1, 3P(-9P^2+9P-3)) = \gcd(P-1, -3P^2+3P-1) = \gcd(P-1, 1) = 1$ and $gcd(g) = \gcd(P+27, P+(3P+3)^3, P+(-3P+3)^3) = \gcd(P+27, 27*(P^2+3P+3), 27*(-P^2+3P-3)) = \gcd(P+27, P^2+3P+3, 6P) = \gcd(P-27, P^2+1P+1, 2) = 1$.
Therefore, by \cref{HBM},
there exist infinitely many distinct prime numbers of the form
$q = f(B, C) \ \text{or} \ g(B, C)$ with $(B, C) \in \mathbb{Z}^{\oplus 2}$ such that $q \equiv 2 \pmod{3}$.
Among them,
we can take distinct $(n-3)/2$ prime numbers $q_{j} = f(B_{j}, C_{j}) \ \text{or} \ g(B_{j}, C_{j})$
($1 \leq j \leq (n-3)/2$)
so that $\gcd(3P, \sum_{j} b_{j}^{-1}c_{j}) = 1$,
where $(b_{j}, c_{j}) := (3PB_{j}+1, 3PC_{j}+1) \ \text{or} \ (3PB_{j}+1, 3PC_{j}+3)$
according to whether $q_{j} = f(B_{j}, C_{j}) \ \text{or} \ g(B_{j}, C_{j})$.
Note that if $P \equiv 0 \pmod{2}$,
then the condition $\prod_{j} b_{j} \equiv 1 \pmod{2}$ in \cref{recipe_odd} holds
for arbitrary $(n-3)/2$-tuple $((b_{j}, c_{j}))_{1 \leq j \leq (n-3)/2}$ taken as above.

For each $(n-3)/2$-tuple $((b_{j}, c_{j}))_{1 \leq j \leq (n-3)/2}$ taken as above,
by \cref{Fermat},
there exist infinitely many prime numbers $l \equiv 2 \pmod{3}$
and a positive even integer $m < p$ such that
$l > \max\{ p, b_{j}, c_{j} \mid j = 1, 2, \dots, (n-3)/2 \}$
and every primitive solution of $x^{3}+P^{\iota}y^{3} = l^{m}z^{n}$ satisfies $x \equiv y \equiv 0 \pmod{l}$.
Therefore, \cref{recipe_odd} implies that the equation
\[
	(X^{3}+P^{\iota}Y^{3}) \prod_{j = 1}^{\frac{n-3}{2}}
		(b_{j}^{2}X^{2} + b_{j}c_{j}XY + c_{j}^{2}Y^{2}) = l^{m} Z^{n}
\]
violates the local-global principle.
The non-singularity follows from the fact that
$q_{j}$ and $q_{k}$ are distinct prime numbers,
hence $[b_{j} : c_{j}] \neq [b_{k} : c_{k}]$ for any $j \neq k$.
The infinitude of the non-isomorphy classes follows from the following \cref{infinitude}.

Next, suppose that $p \neq 3$ and $\iota = 2$.
Let $f(X, Y) = P^{2}(3PX+1)^{3} + (3PY+1)^{3}$ and $g(X, Y) = P^{2}(3PX-1)^{3} + (3PY+3)^{3}$.
Then, since $\gcd(f(0, 0), f(0, \pm 1)) = \gcd(g(0, 0), g(0, \pm1)) = 1$,
we have $\gcd(f(x, y) \mid (x, y) \in \mathbb{Z}^{\oplus 2}) = \gcd(g(x, y) \mid (x, y) \in \mathbb{Z}^{\oplus 2}) = 1$.
%$gcd(f) =(P^2+1, P^2+(3P+1)^3, P^2+(-3P+1)^3) = \gcd(P^2+1, 3P(9P^2+9P+3), 3P(-9P^2+9P-3)) = \gcd(P^2+1, 3P^2+3P+1, -3P^2+3P-1) = \gcd(P^2+1, 3P-2, 3P+2) = \gcd(P^2+1, 6P-11, 4) = 1$
%$gcd(g) = (-P^2+27, -P^2+(3P+3)^3, -P^2+(-3P+3)^3) = \gcd(P^2-27, 27*(P^2+3P+3), 27*(-P^2+3P-3)) = \gcd(P^2-27, 3P+30, 3P-30) = \gcd(P^2-27, P+10, 20) = \gcd(73, P+10, 20) = 1$
The rest part is the same as the above argument.

Finally, suppose that $p = 3$.
Then, we combine $f_{\pm}(X, Y) = P^{\iota}(PX\pm1)^{3}+(PY-1)^{3}$
%P = 3
%$f_{\pm}(X, Y) = 3^{2}(3X\pm1)^{3}+(3Y-1)^{3}$
%$\gcd(f_{+}(0, 0), f_{+}(0, 1)) = \gcd(9-1 = 8, 9+8 = 17) = 1$,
%$\gcd(f_{-}(0, 0), f_{-}(0, 1)) = \gcd(-9-1, -9+8) = 1$.
%P = 6
%$f_{\pm}(X, Y) = 6(6X\pm1)^{3}+(6Y-1)^{3}$
%$\gcd(f_{+}(0, 0), f_{+}(0, 1)) = \gcd(6-1 = 5, 6+125) = 1$.
%$\gcd(f_{-}(0, 0), f_{-}(0, -1)) = \gcd(-6-1= -7, -6-7^3) = 1$,
%The original one is $f_{\pm}(X, Y) = 9(27X \pm 1)^{3}+(9Y-1)^{3}.$
%$\gcd(f_{+}(0, 0), f_{+}(0, 1)) = \gcd(8, 9+8^3) = 1$
%$\gcd(f_{-}(0, 0), f_{-}(0, 1)) = \gcd(-10, -9+8^3) = 1$
to generate $(b_{j}, c_{j})$
%Note that $b_{j}$ is odd if $P = 6$.
so that we can apply \cref{recipe_odd} with a help of \cref{Fermat}.
This completes the proof.
\end{proof}

\begin{lemma} \label{infinitude}
Let $n \in \mathbb{Z}_{\geq 5}$ be an odd integer,
$a \in \mathbb{Z} \setminus \{ 0 \}$.
Let $\mathcal{P} \subset \mathbb{Z}^{\oplus 2}$ be an infinite set of
2-dimensional integral vectors $(b, c)$
such that $[b : c] \neq [b' : c']$ as a rational point of the projective line $\mathbb{P}^{1}$
for each distinct $(b, c), (b', c') \in P$.
For each $\frac{n-3}{2}$-tuple
$\bm{v} = (b_{j}, c_{j})_{1 \leq j \leq \frac{n-3}{2}} \in \mathcal{P}^{\frac{n-3}{2}}$,
let $C(a, \bm{v})$ be the plane curve defined by
\[
	(X^{3}+aY^{3})\prod_{j = 1}^{\frac{n-3}{2}} (b_{j}^{2}X^{2} + b_{j}c_{j} XY + c_{j}^{2}Y^{2}) = Z^{n}.
\]
Then, the set
\[
	\mathcal{C}_{a, n} := \left\{ C(a, \bm{v}) \; \middle| \; \bm{v} \in \mathcal{P}^{\frac{n-3}{2}} \right\}
\]
contains infinitely many non-singular curves non-isomorphic to each other over $\mathbb{C}$.
\end{lemma}

\begin{proof}
Let $C$ be a non-singular curve in $\mathcal{C}_{a, n}$.
Then, since its genus $\frac{(n-2)(n-1)}{2} > 1$,
Schwarz' theorem \cite{Schwarz} ensures that
the automorphism group $\mathrm{Aut}(C)$ of $C$ is a finite group.
\footnote{
	In fact, Hurwitz' theorem \cite{Hurwitz} gives an explicit bound $\#\mathrm{Aut}(C) \leq 84(g-1)$.
	}
In particular, the set
\[
	\mathrm{Quot}(C) := \left\{ (C / \langle \varphi \rangle)^{\mathrm{nb}}
		\; \middle| \; \varphi \in \mathrm{Aut}(C) \right\}
\]
contains finitely many
%possibly isomorphic
$n$-punctured lines
(i.e., $\mathbb{P}^{1}$ minus $n$ points).
Here, $C / \langle \varphi \rangle$ denotes the quotient of $C$
by the cyclic group $\langle \varphi \rangle$ generated by $\varphi$,
and $(C / \langle \varphi \rangle)^{\mathrm{nb}}$ denotes its non-branched locus,
i.e., the image of the points each of whose $\varphi$-orbit consists of
exactly $\#\langle \varphi \rangle$ distinct points.

On the other hand,
for each $\frac{n-3}{2}$-tuple $\bm{v} = (b_{j}, c_{j})_{1 \leq j \leq \frac{n-3}{2}} \in \mathcal{P}^{\frac{n-3}{2}}$,
let $L(a, \bm{v})$ be a punctured line defined by
\[
	(X^{3}+aY^{3}) \prod_{i = 1}^{\frac{n-3}{2}} (b_{j}^{2}X^{2} + b_{j}c_{j} XY + c_{j}^{2}Y^{2}) \neq 0.
\]
Then, we can prove that the set consisting of them
\[
	\mathcal{L}_{a}^{\leq n} := \left\{ L(a, \bm{v})
		\; \middle| \; \bm{v} \in \mathcal{P}^{\frac{n-3}{2}} \right\}
\]
contains infinitely many non-isomorphic $n$-punctured lines.
Indeed, for each $\bm{v} \in \mathcal{P}^{\frac{n-3}{2}}$,
there exist at most $n(n-1)(n-2)$ tuples $\bm{v}' \in \mathcal{P}^{\frac{n-3}{2}}$
such that $L(a, \bm{v}')$ is isomorphic to $L(a, \bm{v})$
because such an isomorphism is extended to an element of $\mathrm{Aut}(\mathbb{P}^{1})$,
which is uniquely determined from the images of three points,
say those satisfying $X^{3}+aY^{3} = 0$.
%The punctured locus may not consist of $n$-points in general.
%At least generically, $n! = \#\mathfrak{S}_{n}$ may be reduced.
As a consequence,
we see that
the set $\mathcal{L}_{a}^{\leq n} / \simeq_{\mathbb{C}}$
contains infinitely many isomorphism classes of $n$-punctured lines over $\mathbb{C}$.

Finally, note that we have a natural injection
\[
	\mathcal{L}_{a}^{\leq n} / \simeq_{\mathbb{C}}
		\ \hookrightarrow
			\left. \left( \bigcup_{C \in \mathcal{C}_{a, n}} \mathrm{Quot}(C) \right) \right/ \simeq_{\mathbb{C}}
\]
induced by $L(a, \bm{v}) \mapsto (C(a, \bm{v}) / \langle \varphi \rangle)^{\mathrm{nb}}$,
where $\varphi([X : Y : Z]) = [X : Y : \zeta_{n}Z]$
with a fixed primitive $n$-th root of unity $\zeta_{n} \in \mathbb{C}$.
Therefore, $\mathcal{C}_{a, n}$ contains infinitely many isomorphism classes of non-singular curves over $\mathbb{C}$
as claimed.
\end{proof}

%%%
%%% example
%%%
\section{Examples}

In this section, we demonstrate that the proof of \cref{main_odd} actually gives
explicit counterexamples to the local-global principle of the form
\[
	(X^{3}+p^{\iota}Y^{3}) \prod_{j = 1}^{\frac{n-3}{2}} (b_{j}^{2}X^{2}+b_{j}c_{j}XY+c_{j}^{2}Y^{2}) = LZ^{n}.
\]

\subsection{degree 7}

First of all,
since the fundamental unit of $\mathbb{Q}(7^{1/3})$ is $\epsilon = 4+2 \cdot 7^{1/3}+ 7^{2/3}$,
\cref{AACM_cubic} is verified for $p = 7$.
Hence, we have $\iota = 1$, and so we can actually take $n = 7$ in \cref{main_odd}.

In order to generate the coefficients $(b_{1}, c_{1}; b_{2}, c_{2})$,
we can use the cubic polynomial $f(X, Y) = 7(3X+1)^{3}+(3Y+1)^{3}$
in the proof of \cref{main_odd}.
Indeed, by \cref{HBM},
the set $f(\mathbb{Z}^{\oplus 2})$ contains infinitely many prime numbers,
for example,
\begin{align*}
71 &= 7 \cdot (3 \cdot 0 + 1)^{3} + (3 \cdot 1 + 1)^{3}, \\
449 &= 7 \cdot (3 \cdot 1 + 1)^{3} + (3 \cdot 0 + 1)^{3}, \\
503 &= 7 \cdot (3 \cdot 11 + 1)^{3} + (3 \cdot (-22) + 1)^{3}, \\
&\ \vdots \quad \text{etc}.
\end{align*}
Among such $(b_{j}, c_{j}) = (3X+1, 3Y+1)$,
we can take,
for example, $(b_{1}, c_{1}) = (1, 4)$ and $(b_{2}, c_{2}) = (4, 1)$.

For each choice of the above coefficients,
we can take $L = l^{m}$
with a prime number $l > \max\{ p, b_{1}, c_{1}, b_{2}, c_{2} \} (= 7)$
and an even integer $m \in \mathbb{Z}_{\geq 2}$
so that every primitive solution of $x^{3}+7y^{3} = l^{m}z^{7}$
satisfies $x \equiv y \equiv 0 \pmod{l}$
(cf.\  condition (6) in \cref{recipe_odd}).
In fact,
as in the proof of \cref{Fermat},
we can generate such $l$ as integral values of another cubic polynomial
$h(A, C) = (21A+1)^{3}+49(21C+1)^{3}$.
For example,
\begin{align*}
262193 &= (21 \cdot 3+1)^{3}+49(21 \cdot 0+1)^{3}, \\
452831 &= (21 \cdot (-2)+1)^{3}+49(21 \cdot 1+1)^{3}, \\
521753 &= (21 \cdot 0+1)^{3}+49(21 \cdot 1+1)^{3}, \\
&\ \vdots \quad \text{etc}.
\end{align*}
Here, we take $l = 262193$ with $(a, c) := (21A+1, 21C+1) = (64, 1)$.

Finally, to generate the exponent $m$,
we use \cref{Fujiwara,quadratic_reduction}.
They ensure that every primitive solution of $x^{3}+7y^{3} = 262193^{m}z^{7}$
satisfies $x \equiv y \equiv 0 \pmod{l}$
whenever $\delta(a^{m}, mc) \equiv \delta(1, m) \equiv 4-2m \pmod{7}$ is a non-quadratic residue.
Thus, we can take $m = 4$.

As a consequence, we obtain an explicit counterexample
to the local-global principle:
\[
	(X^{3}+7Y^{3})(X^{2}+4XY+16Y^{2})(16X^{2}+4XY+Y^{2}) = 262193^{4}Z^{7}.
\]

\subsection{degree 9}

In this case, we can take $p = 3$.
First of all,
since the fundamental unit of $\mathbb{Q}(3^{1/3})$ is $\epsilon = 4+3 \cdot 3^{1/3}+2 \cdot 3^{2/3}$,
we have $\iota = 2$.
Our goal is to obtain the parameters $(b_{1}, c_{1}; b_{2}, c_{2}; b_{3}, c_{3}; L)$
so that the equation
\[
	(X^{3}+7Y^{3})(b_{1}^{2}X^{2}+b_{1}c_{1}XY+c_{1}^{2}Y^{2})(b_{2}^{2}X^{2}+b_{2}c_{2}XY+c_{2}^{2}Y^{2})(b_{3}^{2}X^{2}+b_{3}c_{3}XY+c_{3}^{2}Y^{2})
	= LZ^{9}
\]
defines a non-singular plane curve which violates the local-global principle.

In order to generate the coefficients $(b_{1}, c_{1}; b_{2}, c_{2}; b_{3}, c_{3})$,
we can combine the cubic polynomials
$f_{\pm}(X, Y) = 9(3X \pm 1)^{3}+(3Y-1)^{3}$
as in the proof of \cref{main_odd}.
Indeed, by \cref{HBM},
the sets $f_{\pm}(\mathbb{Z}^{\oplus 2})$ contain infinitely many prime numbers,
for example,
\begin{align*}
17 &= 9 \cdot (3 \cdot 0 + 1)^{3} + (3 \cdot 1 - 1)^{3}, \\
53 &= 9 \cdot (3 \cdot (-1) + 1)^{3} + (3 \cdot 2 - 1)^{3}, \\
233 &= 9 \cdot (3 \cdot 1 + 1)^{3} + (3 \cdot (-2) - 1)^{3}, \\
&\ \vdots \quad \text{etc}
\end{align*}
and
\begin{align*}
71 &= 9 \cdot (3 \cdot 1 - 1)^{3} + (3 \cdot 0 - 1)^{3}, \\
197 &= 9 \cdot (3 \cdot 1 - 1)^{3} + (3 \cdot 2 - 1)^{3}, \\
503 &= 9 \cdot (3 \cdot 0 - 1)^{3} + (3 \cdot 3 - 1)^{3}, \\
&\ \vdots \quad \text{etc}
\end{align*}
respectively.
Among such $(b_{j}, c_{j}) = (3X \pm 1, 3Y+1)$,
we can take them
so that $\sum_{j} b_{j}^{-1}c_{j} \not\equiv 0 \pmod{3}$ (cf.\ condition (4) in \cref{recipe_odd}).
For example, we can take
$(b_{1}, c_{1}) = (1, 2)$, $(b_{2}, c_{2}) = (-2, 5)$, and $(b_{3}, c_{3}) = (2, -1)$
corresponding to $17, 53, 71$ in the above lists respectively.

For each choice of the above coefficients,
we can take $L = l^{m}$
with a prime number $l > \max\{ p, b_{1}, c_{1}, b_{2}, c_{2}, b_{3}, c_{3} \} (= 5)$
and an even integer $m \in \mathbb{Z}_{\geq 2}$
so that every primitive solution of $x^{3}+9y^{3} = l^{m}z^{9}$
satisfies $x \equiv y \equiv 0 \pmod{l}$
(cf.\ condition (6) in \cref{recipe_odd}).
In fact,
as in the proof of \cref{Fermat},
we can generate such $l$ as integral values of another cubic polynomial
$h(A, C) = (3A-1)^{3}+81(3C+1)^{3}$.
For example,
\begin{align*}
431 &= (3 \cdot 3-1)^{3}+81(3 \cdot 0+1)^{3}, \\
647 &= (3 \cdot 0-1)^{3}+81(3 \cdot 1+1)^{3}, \\
773 &= (3 \cdot 2-1)^{3}+81(3 \cdot 1+1)^{3}, \\
&\ \vdots \quad \text{etc}.
\end{align*}
Here, we take $l = 431$ with $(a, c) := (3A-1, 3C+1) = (8, 1)$.

Finally, to generate the exponent $m$,
we use \cref{Fujiwara,quadratic_reduction}.
Indeed,
since $\delta(a^{m}, mc) \equiv \delta((-1)^{m}, 1) \equiv 1+(-1)^{m} \equiv 2 \bmod{3}$ for every even $m$,
we see that every primitive solution of $x^{3}+9y^{3} = 431^{m}z^{9}$
satisfies $x \equiv y \equiv 0 \pmod{l}$
whenever $m = 2$.

As a consequence, we obtain an explicit counterexample
to the local-global principle:
\[
	(X^{3}+9Y^{3})(X^{2}+2XY+4Y^{2})(4X^{2}-10XY+25Y^{2})(4X^{2}-2XY+Y^{2}) = 431^{2}Z^{9}.
\]

\subsection{degree 11}

First of all,
since the fundamental unit of $\mathbb{Q}(11^{1/3})$ is $\epsilon = 1+4 \cdot 11^{1/3} - 2 \cdot 11^{2/3}$,
\cref{AACM_cubic} is verified for $p = 11$.
Hence, we have $\iota = 1$, and so we can actually take $n = 11$.
Our goal is to obtain the parameters $(b_{1}, c_{1}; \dots; b_{4}, c_{4}; L)$
so that the equation
\[
	(X^{3}+7Y^{3})(b_{1}^{2}X^{2}+b_{1}c_{1}XY+c_{1}^{2}Y^{2}) \cdots (b_{4}^{2}X^{2}+b_{4}c_{4}XY+c_{4}^{2}Y^{2})
	= LZ^{11}
\]
defines a non-singular plane curve which violates the local-global principle.

In order to generate the coefficients $(b_{1}, c_{1}; \dots; b_{4}, c_{4})$,
we can combine the cubic polynomials
$f(X, Y) = 11(33X-1)^{3}+(33Y+1)^{3}$ and $g(X, Y) = 11(33X+1)^{3}+(33Y+3)^{3}$
as in the proof of \cref{main_odd}.
Indeed, by \cref{HBM},
the set $f(\mathbb{Z}^{\oplus 2})$ contain infinitely many prime numbers,
for example,
\begin{align*}
39293 &= 11 \cdot (33 \cdot 0 - 1)^{3} + (33 \cdot 1 + 1)^{3}, \\
1265903 &= 11 \cdot (33 \cdot (-2) - 1)^{3} + (33 \cdot 5 + 1)^{3}, \\
3060179 &= 11 \cdot (33 \cdot 2 - 1)^{3} + (33 \cdot 1 + 1)^{3}, \\
&\ \vdots \quad \text{etc}
\end{align*}
and
\begin{align*}
182297 &= 11 \cdot (33 \cdot 2 + 1)^{3} + (33 \cdot (-1) + 3)^{3}, \\
2099927 &= 11 \cdot (33 \cdot (-1) + 1)^{3} + (33 \cdot 4 + 3)^{3}, \\
3281393 &= 11 \cdot (33 \cdot 2 + 1)^{3} + (33 \cdot (-1) + 3)^{3}, \\
&\ \vdots \quad \text{etc}
\end{align*}
respectively.
Among such $(b_{j}, c_{j}) = (33X-1, 33Y+1), (22X+1, 33Y+3)$,
we can take them
so that $\sum_{j} b_{j}^{-1}c_{j} \not\equiv 0 \pmod{3}$
and $\sum_{j} b_{j}^{-1}c_{j} \not\equiv 0 \pmod{11}$ (cf.\ conditions (4) and (5) in \cref{recipe_odd}).
For example, we can take
$(b_{1}, c_{1}) = (-1, 1)$, $(b_{2}, c_{2}) = (67, -63)$, $(b_{3}, c_{3}) = (-67, 166)$,
and $(b_{4}, c_{4}) = (-32, 135)$.

For each choice of the above coefficients,
we can take $L = l^{m}$
with a prime number $l > \max\{ p, b_{1}, c_{1}, \dots, b_{4}, c_{4} \} (= 166)$
and an even integer $m \in \mathbb{Z}_{\geq 2}$
so that every primitive solution of $x^{3}+11y^{3} = l^{m}z^{11}$
satisfies $x \equiv y \equiv 0 \pmod{l}$
(cf.\ condition (6) in \cref{recipe_odd}).
In fact,
as in the proof of \cref{Fermat},
we can generate such $l$ as integral values of another cubic polynomial
$h(A, C) = (33A+1)^{3}+121(33C+1)^{3}$.
For example,
\begin{align*}
1000121 &= (33 \cdot 3+1)^{3}+121(33 \cdot 0+1)^{3}, \\
2507693 &= (33 \cdot (-4)+1)^{3}+121(33 \cdot 1+1)^{3}, \\
4574417 &= (33 \cdot 5+1)^{3}+121(33 \cdot 0+1)^{3}, \\
&\ \vdots \quad \text{etc}.
\end{align*}
Here, we take $l = 1000121$ with $(a, c) := (33A+1, 33C+1) = (100, 1)$.

Finally, to generate the exponent $m$,
we use \cref{Fujiwara,quadratic_reduction}.
They ensure that every primitive solution of $x^{3}+11y^{3} = 1000121^{m}z^{11}$
satisfies $x \equiv y \equiv 0 \pmod{l}$
whenever $\delta(a^{m}, mc) \equiv \delta(1, m) \equiv 8-2m \pmod{11}$ is a non-quadratic residue.
Thus, we can take $m = 6$.
As a consequence, we obtain an explicit counterexample
to the local-global principle:
\begin{align*}
	(X^{3}+11Y^{3})(X^{2}-XY+Y^{2})(67^{2} X^{2}-67 \cdot 63XY+63^{2}Y^{2}) & (67^{2}X^{2}-67 \cdot 166XY+166^{2}Y^{2}) & \\
		\times (32^{2}X^{2}-32 \cdot 135XY+135^{2}Y^{2}) & = 1000121^{6}Z^{11}.
\end{align*}

%%%
%%% numerical
%%%
\section*{Appendix : Numerical verification of \cref{AACM_cubic}}

It is easy to verify \cref{AACM_cubic} numerically for small prime numbers $p$.
In fact, by using Magma \cite{Magma},
we have verified this conjecture for both $P = p \ \text{and} \ 2p$ in the range $p < 10^{5}$.
For example,
the following program returns that there exist no counterexamples
of \cref{AACM_cubic} for $P = p$ in the range $p < 10^{5}$.
Here, recall that \cref{AACM_cubic} holds for $P = p$
if the order $\mathbb{Z}[p^{1/3}]$ of $\mathbb{Q}(P^{1/3})$
has a unit $\alpha + \beta p^{1/3} + \gamma p^{2/3}$
with $\alpha, \beta, \gamma \in \mathbb{Z}$ such that $\beta \not\equiv 0 \pmod{p}$.

\begin{lstlisting}[basicstyle=\ttfamily\footnotesize, frame=single]
> Zx<x> := PolynomialRing(Integers());
> for p in [1..10^5] do;
   > if IsPrime(p) then
   > O := EquationOrder(x^3-p); // create the order O := Z[p^{1/3}]
   > U,phi := UnitGroup(O);
   // U.1 := -1 and U.2 := another generator of the unit group of O;
   > Fpy<y> := PolynomialRing(FiniteField(p));
   > h := hom< O -> Fpy | y >; 
   // represent each element of O mod pO as a polynomial of y = p^{1/3}
      > if Coefficient(h(phi(U.2)), 1) eq 0 then;
      // if the coefficient of p^{1/3} for U.2 is 0 mod p, then
      > <p, h(phi(U.2))>;
      // return such a prime number p and the corresponding U.2
      > end if;
   > end if;
> end for;

<3, 2*y^2 + 2>
\end{lstlisting}

Note that the return $p = 3$ is the conjectural unique exception.

Moreover, the our numerical verification for both $P = p \ \text{and} \ 2p$ in the range $p < 10^{5}$ implies that
there exist non-singular plane curves of degree $n$
which violate the local-global principle
for ``most" odd integers $n$ in the sense of natural density:
Let $\mathbb{N} := \mathbb{Z}_{\geq 1}$,
and $\mathbb{N}^{\mathrm{odd}}$ be the set of positive odd integers.
Set
\begin{align*}
	P &:= \left\{ p : \text{prime number} \right\}, \\
	BP &:= \left\{ p \in P \mid p < 10^{5} \right\}, \\
	M &:= \left\{ n \in \mathbb{N}
		\mid n \not\equiv 0 \pmod{p} \ \text{for all $p \in BP$ and}
			\ n \not\equiv 0 \pmod{p^{2}} \ \text{for all $p \in P$} \right\}, \\
	N &:= \mathbb{N} \setminus (M \cup \{ 1 \}), \\
	\mathbb{N}^{\mathrm{odd}} &:= \{ n \in \mathbb{N} \mid \text{$n$ is an odd integer} \}, \\
	N^{\mathrm{odd}} &:= N \cap \mathbb{N}^{\mathrm{odd}},
\end{align*}
then \cref{main_odd} and the above numerical verification of \cref{AACM_cubic}
(with Poonen's construction \cite{Poonen_genus_one} for $n = 3$)
ensures that
we can construct infinitely many explicit non-singular plane curves of degree $n$
which violates the local-global principle
for each $n \in N^{\mathrm{odd}}$.
Moreover,
if we denote the natural density of $S \subset \mathbb{N}$ by $d(S)$ (if it exists),
then we have
\begin{align*}
	d(M) &= \prod_{p \in BP} (1-p^{-1}) \times \prod_{p \in P \setminus BP} (1-p^{-2})
			= \prod_{p \in BP} (1+p^{-1})^{-1} \times \zeta(2)^{-1} \\
		&< 0.0487529
\intertext{and}
	d(\mathbb{N}^{\mathrm{odd}}) &= \frac{1}{2},
\intertext{hence}
	\frac{d(N^{\mathrm{odd}})}{d(\mathbb{N}^{\mathrm{odd}})}
	&= 1 - \frac{d(M)}{d(\mathbb{N}^{\mathrm{odd}})}
		> 0.90249
\end{align*}
Therefore,
at least $90\%$ of odd integers lie in $N^{\mathrm{odd}}$.

\section*{Acknowledgements}

The authors would like to thank Yoshinori Kanamura
for wonderful introduction of classical counterexamples
to the local-global principle by Selmer, Fujiwara, and Fujiwara-Sudo,
and asking which degrees have explicit counterexamples.
His question was the starting point of our study.
The authors would like to thank also Haruka Saida
for valuable discussions.
The authors would like to thank also Prof.\ Ken-ichi Bannai and the anonymous referee
for their careful reading the draft of this article and giving many valuable comments.
The authors are deeply grateful also for referee's valuable comments and suggestions,
especially for her/his introduction of an excellent textbook by Cohen \cite{Cohen}.
The authors would like to thank also Prof.\ Hiroyasu Izeki, Prof.\ Masato Kurihara, Prof.\ Takaaki Tanaka,
and Dr.\ Yusuke Tanuma for careful reading of the draft of this article and many valuable comments.

\end{document}